\documentclass[12pt]{amsart}

\usepackage[pdftex]{graphicx}

\usepackage[T2A]{fontenc}
\usepackage[utf8]{inputenc}
\usepackage[english]{babel}

\usepackage{amsmath}
\usepackage{amssymb}
\usepackage{amscd}
\usepackage{amsfonts,amsmath,amsthm}

\usepackage[matrix,arrow,curve]{xy}
\usepackage{xfrac}
\usepackage{color}

\newtheorem{thm}{Theorem}

\newtheorem{prop}{Proposition}
\newtheorem{cor}{Corollary}
\newtheorem{lemma}{Lemma}
\newtheorem*{main lemma}{The main lemma}

\newtheorem{dfn}{Definition}

\title[Concurrent normals problem for convex polytopes]{Concurrent normals problem for convex polytopes and Euclidean distance degree }

\author{Ivan Nasonov, Gaiane Panina, Dirk Siersma}

\address{I. Nasonov: StPetersburg State University  wanua-nasonov-i04@yandex.ru; G. Panina: St. Petersburg department of Steklov institute of mathematics
 gaiane-panina@rambler.ru ; D. Siersma: Mathematisch Instituut, Universiteit Utrecht  d.siersma@uu.nl }
\date{\today}

\begin{document}

\begin{abstract}It is conjectured since long that for any convex body $P\subset \mathbb{R}^n$ there exists a point in its interior   which belongs to at least $2n$ normals from different points on the boundary of $P$. The conjecture is known to be true for $n=2,3,4$.

We treat the same problem for convex polytopes in $\mathbb{R}^3$. It turns out that the PL concurrent normals problem differs a lot from the smooth one. One almost immediately proves that a convex polytope in $\mathbb{R}^3$ has $8$ normals to its boundary emanating from some point in its interior. Moreover,  we conjecture that each simple polytope {in $\mathbb{R}^3$} has a point in its interior with $10$ normals to the boundary. We confirm the conjecture for all tetrahedra and triangular prisms and give a sufficient condition for a simple polytope to have
 a point with $10$ normals. 

Other related topics (average number of normals, minimal number of normals from an interior point, other dimensions) are discussed.
\end{abstract}

\maketitle

Consider a compact  convex polytope $P \subset \mathbb{R}^n$, pick a point  $y$ from the interior of $ P$, and  study  normals to the boundary $\partial P$ emanating from $y$. In the smooth case the normals to $\partial P$ are in a natural bijection with the critical points of the \textit{squared distance function } $$|x-y|^2:\partial P\rightarrow \mathbb{R}$$ (abbreviated as $SQD_y$).

Smooth bodies,  their normals, and $SQD_y$ fit nicely in the  Morse theory. The total number of critical points plays a role in the discussion on tight and taut embeddings \cite{Cecil}.
Also the bifurcations of critical points (equivalently, birth-death of normals) played an important role in \cite{Pardon} and \cite{GrPanina}.
Recently the study of the average number of critical points and comparison with the complex case was published in \cite{DH}. It is a part of an ongoing research.

With some effort, Morse-theoretic approach  extends also to polytopes.  So counting normals to the boundary is equivalent to counting critical points of $SQD_y$.

\medskip

In this paper we  address  the following questions.
 Assume $P$ is a compact  convex polytope in $\mathbb{R}^3$.
 
\textbf{Q1:}
What is  the analog of concurrent normals conjecture for convex polytopes? That is, for a given dimension, what is the 
{ minimum over all convex polytopes $P$ of the } 
maximal number of normals emanating from a point $y$ ranging in the interior of $P$? 
\medskip

\textbf{Q2:}  
What is  the minimal number of normals if $y$ ranges in the interior of $P$? {We mean here the minimum  either over all convex polytopes  of a given dimension, or of a given dimension and a given combinatorial type.} 

This minimality question is important because of interest in the  average number of normals  \cite{DH} which  can be interpreted as an indicator of the shape of the poytope, see Section  \ref{SecAverage}.

\begin{dfn}
  For a polytope $P$  and a point $y\in Int \,P$, denote by $n(P,y)$ the number of normals to $\partial P$  emanating from $y$.

  Set also $$\mathcal{N}(P)=\max _yn(P,y),$$
  where $y$ ranges over the interior of $P$.
\end{dfn}

\medskip

Our \textbf{main results} are:
\begin{itemize}
  \item We show that $\mathcal{N}(P)\geq 2n+2$ for $P\subset \mathbb{R}^n$ (Theorem  \ref{Thm2nplus2}).
  \item We propose the \textit{$10$-normals conjecture:} 
$\mathcal{N}(P)\geq 10$  for any convex polytope $P\subset \mathbb{R}^3$.
  \item 
We confirm the conjecture for tetrahedra and triangular prisms (Proposition \ref{PropTetrPrism}) and give   sufficient criteria for a simple polytope $P$
to satisfy $\mathcal{N}(P)\geq 10$  (Theorem \ref{Thm10nicevertex}, Corollary \ref{Cor10normalsCriter},  Corollary \ref{Cor10normalsCriter2}).

  \item We also deduce some result concerning {Q2}  (Theorem \ref{Thm4normalstetrahedron}).
 
\end{itemize}

\section{Definitions and preliminaries}\label{SecDefPrelim}

Let $P\subset \mathbb{R}^n$ be a compact convex polytope with non-empty interior.
A \textit{face }of $P$ is the intersection of $P$ with a support plane.
So the dimension of a face ranges from $0$ to $n-1$.

The faces of dimension $n-1$ are called \textit{facets}. The $0$-dimensional faces are \textit{vertices}, and the $1$-dimensional faces are\textit{ edges}.

 Given a face $F$, assume that the origin $O$ lies in the interior of $F$.  The \textit{inner normal cone} of $F$ is 
  
  $$Cone(F)=\{x\in \mathbb{R}^n|\langle p,x\rangle \geq 0 \ \  \forall p\in P\}.$$

\newpage

\textbf{Normals and active regions}

Throughout the paper we assume that a point  $y$ lies in the interior of $ P$.  It is worthy mentioning that the theory changes a lot if one allows non-convex polytopes and/or points $y$ lying outside $P$.

 Assume there is a support  hyperplane $h$  to the polytope $P$ and a point $z\in h\cap P$ such that $yz$ is orthogonal to $h$.
Then  $yz$ is called a \textit{normal to the boundary of $P$ emanating from }$y$. The point $z$ is called the\textit{ base }of the normal.


Given a polytope $P$ and a face $F$, the \textit{active region} $\mathcal{AR}(F)$ is the set of all points $y\in Int\, P$ such that $F$ contains the base of some normal emanating from $y$. { In other words,  the inner normal cone comes from  the set of inward normal directions to support hyperplanes along $F$.}

\begin{lemma} \label{LemmaAR} Assume that the origin $O$ lies in the interior of a face $F$. Then
   $\mathcal{AR}(F)=\Big(Cone(F)+F\Big) \cap P$.
\end{lemma}

\medskip

\textbf{Bifurcation set and its sheets}

The\textit{ bifurcation set}  $\mathcal{B}(P)$ is the union of the boundaries of active regions. It is the PL counterpart of the focal set that appears in the smooth case. For   all points $y$  not lying on $\mathcal{B}(P)$, the function
$SQD_y$ {should be treated as a Morse function, see Section \ref{SectionMorseTheor}. }

The bifurcation set is piecewise linear. For $P\subset \mathbb{R}^3$, its linear parts not lying on $\partial P$ (\textit{sheets}) fall into two types, red and blue, see the below explanation and Fig. \ref{fig:ActiveRegion}.

\begin{figure}[h]
 \includegraphics[width=0.6\linewidth]{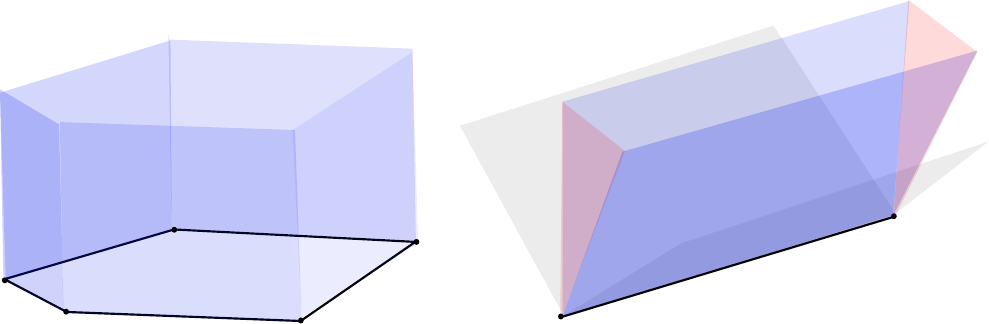}
  \caption { Active regions of a {facet and an edge}}
  \label{fig:ActiveRegion}
\end{figure}

\medskip

We have: 
\begin{enumerate}
  \item The active region of a facet $F$ is the intersection of the infinite  right angled prism  based on $F$ and  the polytope $P$.
  
  The side facets  of the prism are called \textit{ blue  sheets} of  $\mathcal{B}(P)$.
  \item The active region of an edge $E$  sharing two facets $F_1$ and $F_2$ is the intersection of  the polytope $P$ with the wedge-like polyhedron bounded by 
  two planes  containing $E$ and orthogonal to $F_1$ and $F_2$ (these are blue sheets), and another two planes  containing the endpoints of $E$ and orthogonal to $E$ (\textit{red sheets}). 
  \item The active region of a vertex $V$ is the inner normal cone of $P$ at the vertex $V$ (bounded by red sheets)  intersected with the polytope $P$. 
\end{enumerate}

\begin{lemma}\label{PropBirthDeath}
(Birth-death of normals.)

If a point $y$ crosses transversally exactly one sheet of the bifurcation set, either two normals emanating from $y$ die or   two normals   are  born.\end{lemma}

 More precisely, if $y$ crosses transversally a blue  sheet, we have a birth-death  of a normal with the base at a facet, and a normal with the base at an edge.

If $y$ crosses transversally a red  sheet, we have a birth-death of a normal with the base at a vertex, and a normal with the base at an edge.
 \noindent
When  $y$ crosses the sheet, these two normals coincide.

\section{Squared distance function $SQD_y$ and its Morse-theoretic analysis }\label{SectionMorseTheor}

The function $$SQD_y=|x-y|^2:\partial P\rightarrow \mathbb{R}$$ is a non-smooth function defined on a PL sphere.

We apply a construction used in \cite{Heil}:
consider the Minkowski sum $P_{\varepsilon} : =P+ \varepsilon B$, where $B$ is the unit ball and $\varepsilon > 0$. The boundary $\partial P_{\varepsilon}$ is  a $C^1$-manifold, and the normals of $\partial P$ and $\partial P_{ \varepsilon}$ are in a natural bijection: each base $z$ of a normal from $y$ to $\partial P$ corresponds to a base $z_{\epsilon}\in \partial P_ {\varepsilon}$, see Fig. \ref{fig:epsilon}.

There is a natural stratification of $\partial P_{ \varepsilon }$  into $C^{\infty}$-smooth strata. Each face $F$ gives rise to a well-defined stratum $F_{\varepsilon}$ which is a subset of the product of the face $F$ with a sphere in the complementary plane. The  gluing of these strata is directly related to the bifurcations.

\medskip
We will now use the critical point theory of $SQD_y$ on $\partial P_{\varepsilon}$  (with notation $SQD_{y,\varepsilon}$) for the study of normals on $\partial P$.

\begin{figure}[h]
  \includegraphics[width=0.5\linewidth]{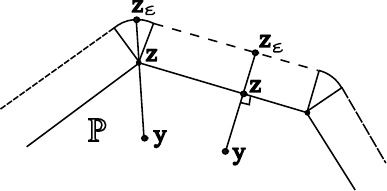}
  \caption { $P_{\varepsilon}$ and the correspondence of critical points}
  \label{fig:epsilon}
\end{figure}

\begin{prop}

{All critical points} of $SQD_{y,\varepsilon}$ on $\partial P_{\varepsilon}$ are the bases of normals to $\partial P_{\varepsilon}$.  

{For $y \notin \mathcal{B}(P)$,} the critical points are  Morse, and the Morse index of a critical point $z\in F_{\varepsilon}$  equals
$ n-1 - \dim F $.

\end{prop}
\begin{proof}
If $y \notin \mathcal{B}(P)$,  the base $z_{\varepsilon}$ of the normal lies in the $C^{\infty}$-part of $\partial P_{\varepsilon}$. Therefore we can use Milnor's index formula: the Morse index is equal to the number of focal points between $z_{\varepsilon}$ and $y$, counted with multiplicity \cite{Milnor}. 
Clearly the focal point on the normal at $z_{\varepsilon}$ is $z$, which is the center of the complementary sphere and has multiplicity $ n-1 - \dim F $.
Since $z$ lies between $z_{\varepsilon}$ and $y$  and there are no other focal points,  the proposition follows.
\end{proof}

\medskip

\textbf{Remark.} {The first statement of the proposition holds true also for points $y$ outside $P$. However, one sees that the Morse index counts are different. For instance, a vertex may be a (local) minimum, etc.}

So counting critical points  of $SQD_{y,\varepsilon}$ amounts to counting normals, exactly as it is in the smooth  case. We will use the machinery of Morse theory:

\begin{prop}\label{PropMorseCount}
\begin{enumerate}  Let  $P\subset  \mathbb{R}^n$, {and assume $y \notin \mathcal{B}(P)$. }

\item The squared distance $SQD_{y,\varepsilon}$ is a Morse function.
  \item We have the Morse Theorem: $$\chi (\partial P)=1+(-1)^{n-1}=\sum_{i=o}^{n-1} (-1)^im_i,$$ where $m_i$ is the number of critical points with Morse index $i$.
  \item Morse inequalities are valid.  In particular we have
  $$m_1 \geq m_0-1.$$
\end{enumerate}
\end{prop}
\begin{proof}

The proof literally repeats that of Milnor \cite{Milnor}, one checks that all the arguments are applicable to our non-smooth case.
 That is,  we have the the regular interval theorem,
the cell attaching statement (in the $C^2$ part), and therefore all the Morse counts.

\end{proof}

This is a Morse-theoretic counterpart of Lemma \ref{PropBirthDeath}:
\begin{prop}\label{PropBifurc}

If $y$ crosses transversally one sheet of $\mathcal{B}(P)$, a pair of critical points of  $SQD_{y,\varepsilon}$ either dies or is born.
In $\mathbb{R}^3$ there are  simple rules:
\begin{enumerate}
  \item Crossing a blue sheet amounts to death-birth of a minimum and a saddle.
  \item Crossing a  red sheet amounts to death-birth of a maximum and a saddle.
\end{enumerate}

\end{prop}

\begin{lemma}
Let $y \in \mathcal{B}(P)$, then arbitrary close to $y$ there is  a point $y' \in P \setminus \mathcal{B}(P)$ such that $n(P,y') \ge n(P,y) $.  Therefore
$\mathcal{N}(P)$ can be computed, using only points $y \notin \mathcal{B(P)}$, and $\mathcal{N}(P)$ is always even.
\end{lemma}
\begin{proof}
  The statement is clear for a point lying on a unique sheet.
If several sheets are involved, choose a line $l$ through $y \in \mathcal{B}(P)$, which is transversal to all sheets.
Passing through $y$ in one direction will either increase or decrease the number of normals by $1$ for each of the involved sheets. {So $$n(y,P)=\frac{n(P,y')+n(P,y'')}{2},$$ where $y'$ and $y''$ are nearby points lying on the line $l$ on different  sides with respect to $y$. So either $ n(P,y') \ge n(P,y)$, or $n(P,y'') \ge n(P,y)$.}

 The number $n(P,y)$ is even for a generic $y$ since the number of critical points of a Morse function on a closed manifold is even.
\end{proof}

\medskip

{Since normals to $\partial P_{\varepsilon}$ are in bijection with  normals to $\partial P$,  we shall use the Morse theory for
  $\partial P_{\varepsilon}$, keeping in mind this correspondence. We will say  ''a critical point is
attained at a face of $P$''.}

\section{First results: polygons, tetrahedra in $\mathbb{R}^3$ and polytopes in $\mathbb{R}^n$.}

For a polytope $P$, its face $F$ and a point $y\in Int\, P$,  say that $y$ \textit{projects to} $F$, if
the orthogonal  projection of $y$ to the affine hull $\mathrm{aff }(F)$ belongs to $F$.

\subsection{Polygons in  $\mathbb{R}^2$}

Let $P$ be a planar convex polygon.
      If a point $y$ lies outside the bifurcation set $\mathcal{B}(P)$, all the local minima of $SQD_y$ are attained at the edges, and (local) maxima are attained at vertices.

An edge  contributes a local minimum iff    $y$ projects to the edge.

 \begin{prop}
   \begin{enumerate}
     \item $\mathcal{N}(P) \geq 6$ for every convex polygon.
     \item  $\mathcal{N}(P) = 6$  for all triangles.
   \end{enumerate}
 \end{prop}
 \begin{proof}
   (1) Take a circle inscribed in $P$. Its center $y$ contributes at least three minima since there are at least three tangent points. Therefore there are at least three maxima.
   
   (2) is now clear since between two local minima there is always exactly one point of local maximum.
 \end{proof}

 \textbf{Example.} Figure \ref{fig:Normals-obtuse}  depicts an obtuse triangle. Each point in the shaded domain has $3$ maxima and $3$ minima of $SQD_y$.
 Points outside the domain have $2$ maxima and $2$ minima.

\begin{figure}[h]
 \includegraphics[width=0.5\linewidth]{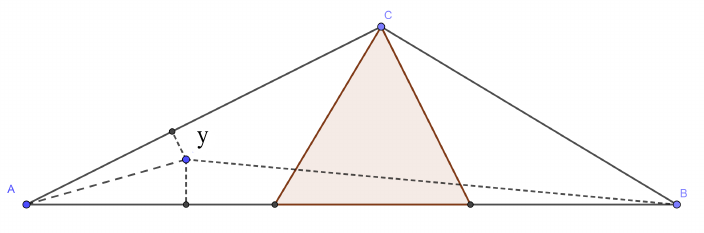}
  \caption { An obtuse triangle  and a point with $4$ normals. }
  \label{fig:Normals-obtuse}
\end{figure}

\subsection{Convex polytopes in  $\mathbb{R}^3$}

 Let $P\subset\mathbb{ R}^3$ be a compact convex polytope.

Local minima of $SQD_y$ are attained at  (some of the) facets,
saddles are attained at edges, and
local maxima are attained at vertices.

\begin{lemma}\label{LemmaSad} Let   $P\subset \mathbb{R}^3$ be a convex polytope, and let the dihedral angle at its edge $e$ be acute. Then for a  point $y\notin \mathcal{B}(P)$, the function $SQD_y$ attains a saddle  at the edge $e$ iff $y$ projects to $e$.
\end{lemma}
\begin{proof}
Follows  from Lemma \ref{LemmaAR}.
\end{proof}

\begin{lemma}\label{LemmaSaddles} For a convex $P\subset \mathbb{R}^3$ and a  point $y\notin \mathcal{B}(P)$, we have
$$n(P,y)=2+2s, $$  where $s$ is the number of saddles.
\end{lemma}
\begin{proof}
Follows  from Proposition \ref{PropMorseCount} and $n(P,y)=m+M+s,$ where $M$ is the number of maxima, and $m$ is the number of minima.
\end{proof}

\begin{thm}\label{Thm2nplus2}

\begin{enumerate}For any convex polytope $P\subset \mathbb{R}^n$,
                                   \item There is always a point in $P$ with at least $2n+2$ normals  to $\partial P$. 
                                   
                                   In other words, $$\mathcal{N}(P)\geq 2n+2.$$
                                   \item $\mathcal{N}(P)$ never exceeds the number of faces of $P$.
                                 \end{enumerate}

\end{thm}

\begin{proof}
  (1) Let $y$ be the center of the inscribed sphere (the sphere of maximal radius contained in $P$).

We may assume that there are at least $n+1$ tangential points  with  the facets of $P$. This is true if the inscribed sphere is unique. If not, there exists an   inscribed sphere with this property.

 Each of the  tangential points contributes a minimum, each of them lies strictly inside some facet, and therefore survives under small perturbations of the center $y$. So if  $y\in \mathcal{B}(P)$ (and therefore $SQD_y$ is not a Morse function), perturb $y$ and get a Morse function with at least $n+1$ local minima. Besides, there is at least one maximum, and, by Proposition \ref{PropMorseCount}, there are at least $n$ other critical points with Morse index $1$.   

(2) Each face contributes at most one critical point. 
\end{proof}

\begin{thm}\label{Thm10normalstetrahedron}  
                  There exists  tetrahedra with $\mathcal{N}(P)=10, 12,$ and $14$.
\end{thm}

\begin{proof}
 (1) For the center $y$ of  a regular tetrahedron one clearly has $n(P,y)=14$: each face contributes a critical point ($4$ maxima, $4$ minima and $6$ saddles). 

\bigskip 

(2) Now prove the existence of a tetrahedron with $\mathcal{N}(P)=10$.  Take the following  quadrilateral $ABCD$ lying in the horizontal plane as in Fig. \ref{fig:Tetr10}. Perturb it a little bit by shifting $B$ upwards from the plane. Now we have a spatial tetrahedron $ABCD$.
\begin{description}
  \item[i] No point $y\in Int\, ABCD$  has saddles on both $AD$ and $BC$ simultaneously.
   Indeed,  no point from the tetrahedron projects orthogonally to both of the edges (due to the chosen quadrilateral).
  \item[ii] If $y$ has saddles on both $AC$ and $BD$, then $y $  has a saddle neither on $AD$, nor on $BC$.
   Indeed, the dihedral angles at these two edges are very obtuse (can be arbitrarily close to $\pi$), and therefore   the active regions
  almost degenerate to planes. So their intersection is a neighborhood of a line that is transverse to the horizontal plane.
\end{description}
One concludes that the number of saddles does not exceed $4$. Lemma \ref{LemmaSaddles} completes the proof.

 \bigskip

 (3) { A generic tetrahedron with $\mathcal{N}=12$  can be obtained in a similar way if one  follows the above construction for a polygon $ABCD$  such that the intersection of its diagonals projects to exactly three of its edges.}
\end{proof}

\bigskip

\begin{figure}[h]
 \includegraphics[width=0.6\linewidth]{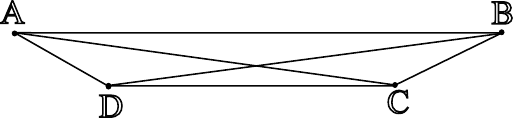}
  \caption {A tetrahedron with at most $10$ normals from an inner point.}
  \label{fig:Tetr10}
\end{figure}

\bigskip
\newpage

\begin{thm}\label{Thm4normalstetrahedron}
\begin{enumerate}
                  \item  For a tetrahedron $P$ and a point $y\in Int\, P$,\newline  $n(P,y)$ is at least $4$.

                  \item There exists a tetrahedron $P$ and a point $y$  with $n(P,y)=4$.
                \end{enumerate}

\end{thm}
\bigskip
\textbf{Remark.} Statement (1) of the theorem is no longer valid for arbitrary convex polytopes. It is easy to cook a
polytope $P$ and a point $y$ inside $P$ (lying very close to a facet) with $n(P,y)=2$.

\begin{proof}

(1).  Assume the contrary: for $P=ABCD$ and some point  $ y$ there is no saddle point of the $SQD_y$. Then there is exactly one minimum and 
exactly one maximum.  Assume that the minimum point $z$ of the $SQD_y$ belongs to the facet $ABC$.
The point $z$ projects to at least two edges of the triangle $ABC$. Let them be $BA$ and $CA$.
\medskip

We prove a number of statements, using lemmata \ref{LemmaSad} and \ref{LemmaSaddles}.

 \textbf{(a):} The dihedral angles at the edges  $BA$ and $CA$ are obtuse. The dihedral angle at $BC$ is acute. 

Indeed, otherwise  the edges  $BA$ and $CA$ contain saddle points of the $SQD_y$.  Besides, the dihedral angles at $AB$, $AC$ and $BC$ cannot be all obtuse.

 \textbf{(b):} the angle $BAC$ is acute.

 Indeed, otherwise, $z$ projects to $BC$ as well, and there is a saddle point on $BC$.

\textbf{(c): } The projection of the vertex  $D$ to the plane $ABC$ lies in the shaded domain on the left, see Fig. \ref{FigProj}.
Therefore the angles
$\angle DAB$ and $\angle DAC$ are obtuse.

  \begin{figure}[h]
 \includegraphics[width=0.6\linewidth]{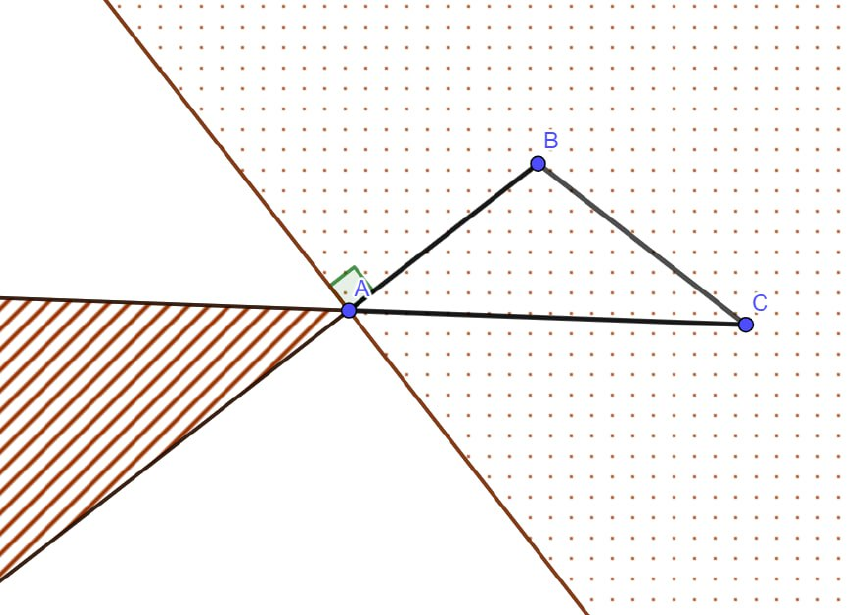}
  \caption {An illustration for (c).}
  \label{FigProj}
\end{figure}

\textbf{(d):} One of the two dihedral angles at the edges $DC$ and $DB$ is obtuse.

Indeed, otherwise all dihedral angles at the edges of the face $DBC$ are acute, which means that
all the points of $P$ (including  the point $y$) project to $DBC$. A contradiction to the uniqueness of minimum. 

\medskip

So without loss of generity we may assume that the dihedral angle $DC$ is obtuse. Then, at the vertex $C$ there are two obtuse dihedral angles at the edges $DC,AC$, and an acute planar angle $DCA$ (since $\angle DAC $ is obtuse). It follows from here that $\angle DCB$ and $\angle ACB$ are obtuse (we have already used such fact above, at the vertex $A$).

We are going to show that  there is a saddle point of $SQD_y$ on $DB$. This follows from statements (e) and (f) below.

\textbf{ (e):} the dihedral angle at $DB$ is acute.

Let the dihedral angle $DB$ be obtuse. Then the planar angle $DBC$ is obtuse, and the triangle $DBC$ has two obtuse angles. A contradiction.

\textbf{ (f): }  $y$ projects to $DB$.

Consider the plane $b$ (respectively, $d$) through $B$ (respectively, $D$) which is perpendicular to $DB$. Since $\angle DCB$ and $\angle DAB$ are obtuse, all the angles $BDC, BDA, DBC,DBA$ are acute. It means that
all the points of $P$ (including  the point $y$) lie between $b$ and $c$, and therefore, project to $DB$.

(2) Here is an example: consider the tetrahedron $ABCD$ with following properties:
\begin{description}
	\item[i] The dihedral angles at the edges $AC,DC,DB$ are obtuse, and the others are acute,
	\item[ii] The planar angles $ADB, ACB, DCB, ADC$ are obtuse, others are acute.
 \end{description}

Such a tetrahedron exists, although it is a very subtle one: the obtuse angles are almost right angles, so Figure \ref{FigTetr4} is unfortunately not much instructive. Set, for example,  $A=(-5;0;0), \ B=(2;0;0),\ C=(-1,54;-2,02;0), \ D=(-3,87;-1,11;1,52)$.

There exists a point $y$ inside $ABCD$, close to the vertex $C$ such that: the angles between the three pairs of planes $ (yDB, DBA), \  (yDC, DCB) ,$ and $(yAC, ACD)$  are obtuse. 

The $SQD_y$  has (the unique) saddle at the edge $AB$. Indeed, there are no saddles at the edges $DB, DC, AC$ since the above  angles are obtuse. Additionally, we can request $yCB$ and $yDA$ be obtuse, because such condition doesn't contradict the condition above. So, there are no saddles at $AD$ and $CB$.

In this example the function  $SQD_y$  has two local maxima  attained at the vertices $A$ and $B$, and one minimum at the face $ABC$.

  \begin{figure}[h]
 \includegraphics[width=0.8\linewidth]{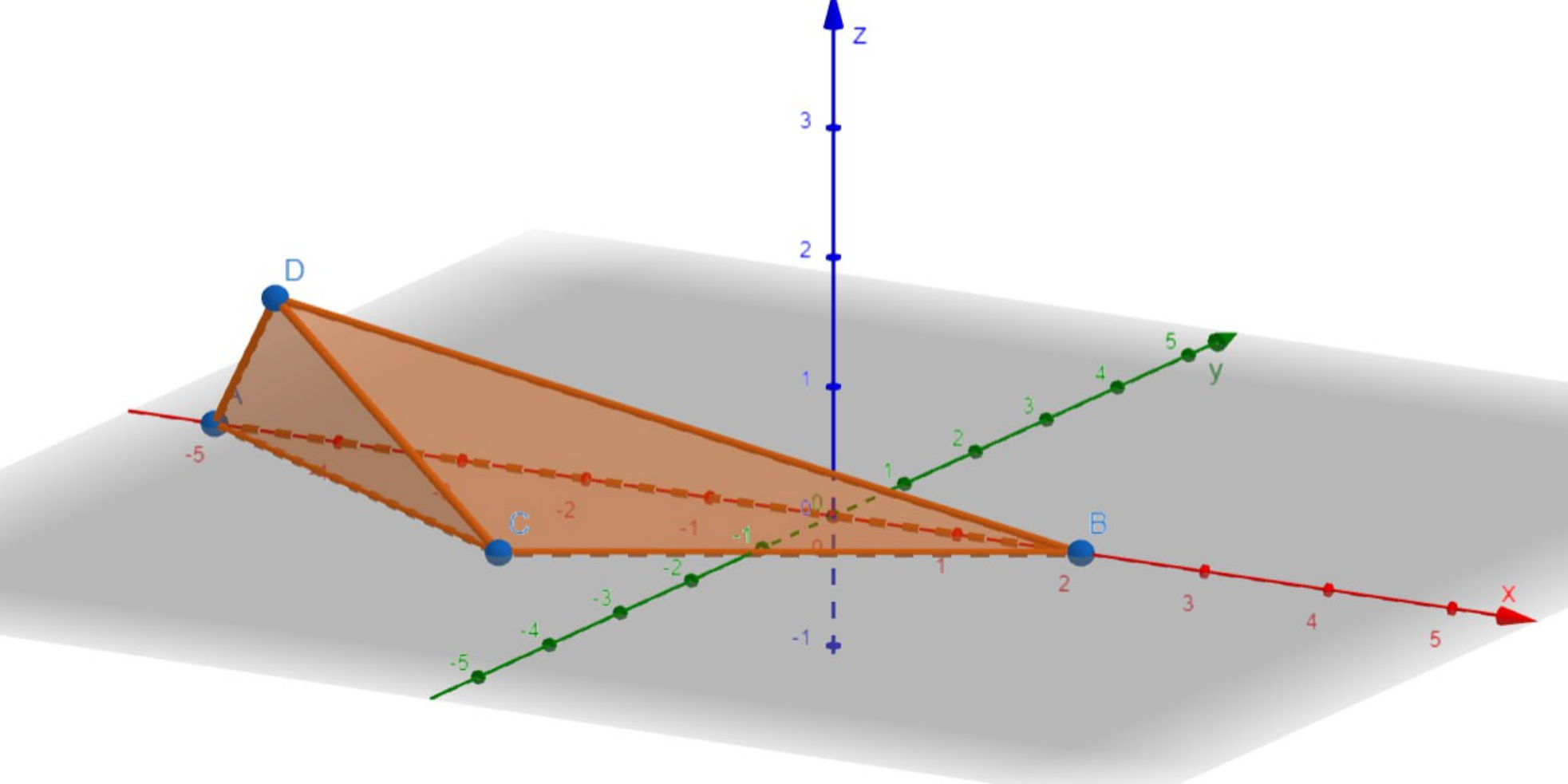}
  \caption {A tetrahedron $P$ with a point $n(y,P)=4$.}
  \label{FigTetr4}
\end{figure}

\end{proof}

\newpage

\section{A toolbox: trihedral angles and spherical triangles}

We need to develop a spherical counterpart  of the concurrent normals and $SQD_y$ theory. 

A \textit{spherical triangle}  is a triangle with geodesic edges lying in the standard sphere $S^2$. It is always supposed to fit in an open hemisphere.

For points $X,Y \in S^2$ denote by  $|XY|$ the spherical \textit{distance} between the two points.

Say that   $X$ \textit{projects to }a geodesic segment $s$ connecting $Y$ and $Z$, if the segment and the point $X$ fit in an open hemisphere, and $s$ contains a point  $W$ such that the geodesic segment $XW$ is orthogonal to the segment,{ and $|XW|$ realizes the Hausdorff distance between the point  $X$ and the segment $s$.  So the set of points $X$ projecting to $s$ is the intersection of two hemispheres.}

\begin{dfn}\label{dfnNice}
  A spherical triangle $ABC$ is \textit{nice}  if there is a point $X$ in its interior such that

(a)  $X$ projects to all the three edges of the triangle, and

(b) $|XA|<\pi /2$, $|XB|<\pi /2$, and $|XC|<\pi /2$.

Otherwise a spherical triangle is called \textit{skew}.
\end{dfn}

\begin{lemma} \label{LemmaNice} A spherical triangle $ABC$ is skew iff (up to relabeling of the vertices)\begin{enumerate}

                                                                                                 \item $|CA|> \pi/2$
                                                                                                 \item $|BC|< \pi/2$
                                                                                       
                                                                                           \item $\angle ABC> \pi/2$      
                                                                                                 \item $\angle BAC <\pi/2$
                                                                                                 \item $ \angle BCA < \pi/2$
                                                                                                 
                                                                                                 \item $|BA|> \pi/2$

                                                                                  \item $|AZ| > \pi/2$ for point $Z\in AC$ such that $ZB$ is orthogonal to $BC$.
                                                                                               \end{enumerate}
\end{lemma}

\begin{proof}  Assume $ABC$ is skew.
  Take the smallest circle $Circ$  lying in the sphere that contains $ABC$. Two cases are possible:
 \begin{description}     
    \item[i] All the three vertices $A, B, C$ lie on $Circ$. Set $X$  be the center of the circle. The  (spherical) distance  from $X$ to the
     boundary of the triangle has three maxima (at the vertices $A, B, C$), therefore also three minima.
        This means that $X$ projects to the three edges. 
        
        Besides, the radius of $Circ$ is smaller than $\pi/2$, so the triangle is nice, and we are done in this case.
    
    \item[ii] $Circ$ contains only two vertices, say, $A$ and $C$. Then the center of $Circ$  is the  centerpoint $O$ of the edge  $AC$. 
    
     (1).
   Take a point  $X$ on the bisector of the angle $ABC$  lying very close to the edge $AC$. $X$ projects to all the three edges, and since $ABC$ is skew,
   its distance to either $A$ or $C$ is greater than $\pi/2$. So we may assume that $|AC|>\pi/2$.

\begin{figure}[h]
  \includegraphics[width=0.6\linewidth]{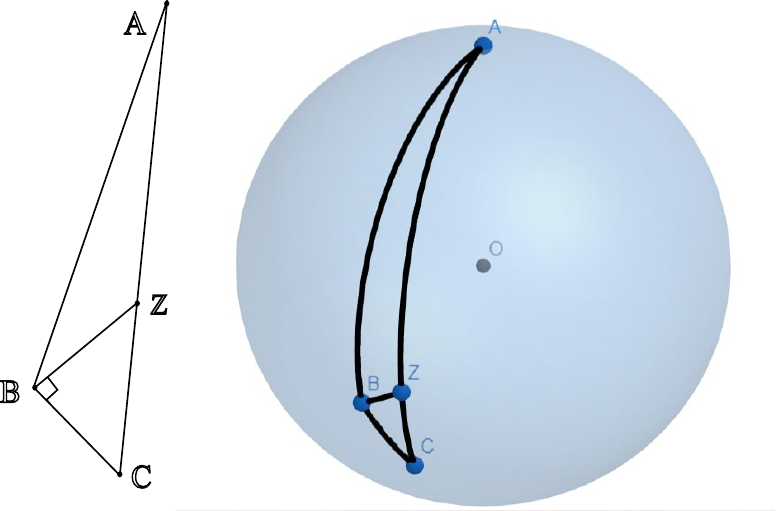}
  \caption { A skew triangle}
  \label{fig:skew}
\end{figure}
    
    (2).
    Take a point $X$ lying inside the triangle, very close
    to  $O$. The condition (b) from Definition \ref{dfnNice} is satisfied. Besides, $X$ projects to $AC$.  The points $A$ and $C$ are maxima of the distance, so 
    $X$ projects to one of the other edges, say, to $AB$, but not to $BC$. Then $B$ and $C$ belong to one and the same  octant  of the sphere obtained by the following construction (see Fig. \ref{FigOctant}):
     the triangle $ABC$ lies in a hemisphere  with a pole at $O$. $AC$ cuts off a quarter of the sphere containing $ABC$. The orthogonal (to $AC$) line passing through $O$ cuts off an octant, see Fig. \ref{FigOctant}. 
     
     \begin{figure}[h]
 \includegraphics[width=0.6\linewidth]{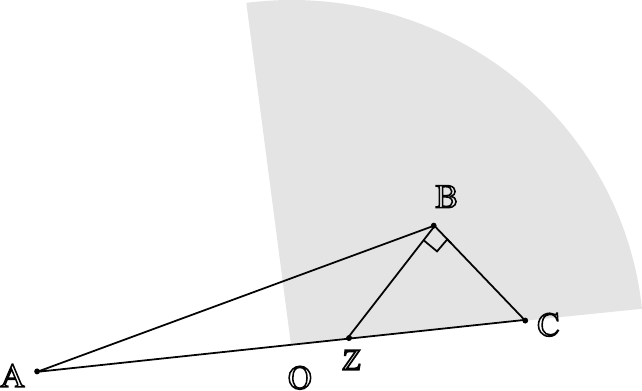}
  \caption {The octant.}
  \label{FigOctant}
\end{figure}

    All distances in this octant don't exceed $\pi/2$, so $|BC|<\pi/2$.

    (3) is true because $B$ lies inside the circle $Circ$.
    
    (4) and (5) follow  directly from (3).

    (6). If $|AB|<\pi/2$, then
     a point $X$ lying on bisector of the angle $ABC$ in the interior of the triangle, very close to $B$, satisfies (a) and (b) from
     Definition \ref{dfnNice}.

     (7)  Pick a point $X$ in the interior  of $ABC$, inside the (spherical) active region of the edge $BC$, and very close to $Z$. The edges contribute three minima of  spherical $SQD_y$,
        so  if $ABC$ is skew then $|XA|>\pi/2$, since all other conditions of the definition are fulfilled (point $X$ lies inside the octant from (6)). Therefore, $|ZA|>\pi/2$.

    Now let's proof the converse.
    Let $ABC$ satisfy the conditions (1) --- (7), and a point $X$ exists. Then $X$ lies inside the (spherical) active region of $BC$, which is bounded by $BZ$. But for every point $X$ inside it, the distance $|AX|>\pi/2$. Indeed, this is true for $B$ and $Z$. Besides, the angle $AZB$ is obtuse,  so $|AX|$ is monotone on the segment $BZ$. Therefore, $ABC$ is skew.

  \end{description}
\end{proof}

\begin{dfn}
  Let $A\subset S^2\subset \mathbb{R}^3$ be a spherically convex set  contained in a hemisphere. Its polar dual is defined as
  $$A^\circ = \{x\in S^2|\langle x,y\rangle \geq  0 \ \ \  \forall y\in A.\}.$$
\end{dfn}

\begin{prop} \label{prop-skew-dual}
  A spherical triangle is skew iff its polar dual is skew.
\end{prop}

\begin{proof}
 Duality turnes edges to vertices and vice versa. An edge length $\alpha$ turnes to the angle  $\pi-\alpha$. Basing on Lemma  \ref{LemmaNice}, let us check that conditions (1) --- (7) are self-dual. It's quite clear for (1) --- (6). Condition (7) is equivalent to folowing: $\angle CBZ'>\pi/2$ for point $Z'\in AC$ such that $AZ' = \pi/2$. This condition turns to (7) for the dual triangle.   
\end{proof}

 Let $P\subset \mathbb{R}^3$ be a simple polytope, { that is, each vertex has exactly three emanating edges. To each of its} vertices $V$ is associated a spherical triangle $Tr(V)$
 which is the intersection of $P$ with a small sphere centered at $V$.  We equip the small sphere with the metric of the standard sphere (with the unit radius).

The edges (resp., vertices) of $Tr(V)$  correspond to facets (resp., edges) of $P$ incident to $V$.

Let $y$ be a point lying in $P$ very close to $V$. Denote by $Y$ the intersection of the ray $Vy$ with the sphere.
The critical points of $SQD_y$  can be easily translated into the language of spherical geometry of $Tr(V)$:

\begin{lemma}\label{LemmaCriterionMaxMin Sadd}
  In the above notation, \begin{enumerate}
                          \item $V$ is a local maximum of   $SQD_y$  iff spherical distances
                          $|AY|$, $|BY|$, and $|CY|$  are smaller than $\pi/2$, where $ABC$ is the spherical triangle $Tr(V)$.
                          \item Let $e$ be an edge of $P$ incident to $V$. Assume  $A= Tr(V)\cap e$. The edge $e$ contributes  a saddle of   $SQD_y$  iff \begin{enumerate}
                                                                                           \item $A$ is a local maximum of $|Yx|^2$, where $x$ ranges over the boundary of the spherical triangle $ABC$.
                                                                                           \item $|AY|<\pi/2$.
                                                                                         \end{enumerate}
                          \item A  facet of $P$ incident to $V$ contributes  a minimum of   $SQD_y$  iff \begin{enumerate}
                                                                                           \item $Y$ projects to the corresponding edge (let it be $AB$).
                                                                                           \item The spherical Hausdorff distance between $Y$ and $AB$ is less than $\pi/2$.
                                                                                         \end{enumerate}
                        \end{enumerate}
\end{lemma}

\section{The $10$-normals conjecture for simple polytopes in $\mathbb{R}^3$}

\bigskip
\textbf{Genericity assumption}

In the section we restrict ourselves by \textit{generic simple polytopes}. This means that
\begin{description}
  \item[i] There are no right dihedral angles at the edges, and also no right planar angles of the facets.
  \item[ii] The sheets of $\mathcal{B}(P)$ intersect transversally.
\end{description}

\begin{dfn}\label{dfnNiceVert}
A vertex $V$ of a simple polytope $P$ is nice if $Tr(V)$ is a nice triangle.
\end{dfn}

\begin{thm}\label{Thm10nicevertex}
  If a simple convex polytope $P$ has at least one nice vertex, then  $\mathcal{N}(P)\geq 10$.
\end{thm}
\begin{proof}

Assume that $V$ is a nice vertex of a  polytope $P$. Take the small sphere centered at $V$ and the spherical triangle $Tr(V)$. Since the triangle is nice, Definition  \ref{dfnNice} guarantees the existence of a point $X$ with special properties.

Take the  ray  $VX$.  It belongs to the active region $\mathcal{AR}(V)$, that is, for every point on the ray, $V$ is a local maximum.
Take a point $y$ lying on the ray very close to $V$. 
By Definition \ref{dfnNice}  and Lemma \ref{LemmaCriterionMaxMin Sadd},  $y$ belongs to active regions of all three edges and all three facets incident to $V$.

Now let the point $y$ travel along the ray starting at $V$.
 Initially $SQD_y$ has (at least) three minima $m_1,m_2,$ and $ m_3$, lying on the faces incident to $V$, three saddles $s_1,s_2,$ and $ s_3$, one local maximum $V$, which cannot be a global one. Therefore there is also a global maximum $V'\neq V$. Altogether we have at least $8$ critical points of $SQD_y$, or, equivalently, $8$ normals  emanating from $y$.

Keep moving until the  first bifurcation  or until  the point $y$ leaves the interior of $P$. Let us analyze  the first event.  \begin{enumerate}

\item The point $y$ leaves the polytope $P$. This means that $y$ hits one of the facets which is not incident to the vertex $V$.
Just before  $y$ leaves the polytope $P$, there are  (at least) $4$ minima and $2$ maxima, therefore at least $4$ saddles.
\item A birth of two new critical points. This gives us $10$  normals.

                                           \item One of the saddles $s_i$ dies. This happens in two cases:
                                           \begin{enumerate}
                                             \item $s_i$ meets a (local) maximum of $SQD_y$. The maximum cannot be $V$  since by construction, $s_i$ slides away from $V$.

                                             Let us prove that the maximum is not  $V'$. Indeed, recall that  $s_i$ is the orthogonal projection of $y$ to the edge.  If $s_i$ meets $V'$  before $y$ leaves $P$, then $|Vy|>|VV'|$. A contradiction.

                                              So just before the bifurcation $SQD_y$  has $3$ minima and $3$ maxima, which proves the desired.
                                             \item $s_i$ meets a minimum. It can be none of $m_1,m_2$, and $m_3$, so just before the bifurcation $SQD_y$  has $4$ minima and $2$ maxima, which proves the claim.
                                                 
                                                 \item The point $V'$  meets a saddle point and dies.  This happens if $V'$ meets one of the saddles. This cannot be one of $s_1,s_2,s_3$ (already discussed above). So just before the bifurcation $SQD_y$  has $4$ saddles, which proves the claim.
                                           \end{enumerate}

                                         \end{enumerate}

\end{proof}

\bigskip

It is still an open question if there exist a polytope with $\mathcal{N}(P) = 8$.
But we anyhow have the following:

\begin{cor}\label{Cor10normalsCriter}
Assume $P$ is a simple polytope, $\mathcal{N}(P)<10$.  Then

\begin{enumerate}
  \item Each vertex has  exactly  two incident acute dihedral angles.
  \item Each vertex has  exactly  one incident acute planar angle.  This angle is not between two  edges that give acute dihedral angles.
\end{enumerate}
\end{cor}

\begin{cor} 
Assume that $\partial P$  lies between two concentric  spheres whose radii relate as $\sqrt{2}$.
Then $\mathcal{N}(P)\geq 10$.
\end{cor}
\begin{proof}
  It's easy to see that in this case there are no acute dihedral angles in $P$. A contradiction to Corollary $\ref{Cor10normalsCriter}$.
\end{proof}

\begin{cor}\label{Cor10normalsCriter2}
Assume that $\mathcal{N}(P)< 10$. Then the outer normal fan of $P$ yields a tiling of the sphere $S^2$ into skew triangles.
\end{cor}

\begin{proof}
  By Proposition $\ref{prop-skew-dual}$, a dual to a skew triangle is also a skew triangle. It remains to remind that the outer normal fan
  tiles the sphere into spherical polytopes dual to $Tr(V)$, where $V$ ranges over the vertices of $P$. 
\end{proof}

\begin{prop}\label{PropTetrPrism}\begin{enumerate}
             \item For any triangular prism $P$, $\mathcal{N}(P)\geq 10$. 
             \item For any tetrahedron $P$, $\mathcal{N}(P)\geq 10$. 
           \end{enumerate}

\end{prop}

\begin{proof}
  (1) Assume the contrary. This means that each of the six vertices is skew, and therefore, there are exactly six acute planar angles of the facets of the prism. However, each triangular facet of a prism has at least two acute  angles, whereas each quadrilateral facet has at least one. Altogether there  at least seven acute angles. A contradiction.

  (2) We give two proofs:
  
  (a) A similar count of acute planar angles as in (1) proves the statement.
  
  (b)The outer normal fan  of a tetrahedron tiles the sphere into four triangles with vertices, say, $A,B,C,D$. If $\mathcal{N}(P)< 10$, then all these triangles are
  skew. Take the longest of their edges, say, $AB$. By Lemma  \ref{LemmaNice}, the vertices $C$ and $D$  (as well as $A$ and $B$) lie in the hemisphere centered at the midpoint of $AB$. A contradiction.
\end{proof}

\bigskip
Corollary \ref{Cor10normalsCriter2} hints that for
the majority of simple polytopes, $\mathcal{N}(P)\geq 10$. However, the problem which is still open: \newline
\textit{is $\mathcal{N}(P)\geq 10$ always true for simple polytopes?}

\section{Averaging normals, Euclidean Distance Degree}\label{SecAverage}

Following the idea of Euclidean distance degree for smooth bodies  or algebraic subsets of Draisma et al  \cite{DH},
let us define for a convex compact polytope the average number of normals:

$$\mathcal{EN}(P) = \frac{1}{Vol(P)}\int_{P} n(P,y)dy $$

The bifurcation set $\mathcal{B}(P)$ divides $P$ into chambers $\sigma$  such that
 the number $n(P,y)$ is constant on each of the chambers   (we use the notation $n(P,\sigma)$).

We immediately have an upper bound:

\begin{prop}
$$\mathcal{EN}(P)  \le \frac{1}{Vol(P)} \mbox{\rm \Big(Total number of faces\Big)}. $$
\end{prop}

\begin{prop}
$$ \mathcal{EN}(P) = \frac{ \sum_{\sigma} Vol(\sigma ) n(P,\sigma)}{Vol(P)}  $$
\end{prop}
This average contains information about the shape of $P$ which seems to be related to obtuse angles.

\begin{prop}
  
For a triangle in the plane,\begin{enumerate}
                              \item $ \mathcal{EN}(P)$ equals $6$ iff the triangle is not obtuse.
                              \item For an obtuse triangle, $ \mathcal{EN}(P)$ is between $4$ and $6$, depending on the value of the obtuse angle. It tends to $6$ if the obtuse angle tends to $\pi/2$, and tends to $4$ if the obtuse angle tends to $\pi$.
                            \end{enumerate}

\end{prop}

\begin{prop}
  
For a tetrahedron,\begin{enumerate}
                              \item $4< \mathcal{EN}(P)\leq 14$.
                              \item    $ \mathcal{EN}(P)= 14$  iff all planar angles of the faces and all the dihedral angles are acute.
                            \end{enumerate}

\end{prop}
\begin{proof}  (1) follows from Theorem \ref{Thm4normalstetrahedron} and \ref{Thm2nplus2}.

 $ \mathcal{EN}(P)= 14$  iff each inner point $y$ of the tetrahedron has $14$ emanating normals.  This happens iff $y$ projects to all the edges, and all the dihedral edges are acute.
\end{proof}
Since $\mathcal{EN}(P)$  measures the overlaps between active regions it depends on the shape of the polytope within the  given the combinatorial type,
although
the concept of shape is not well-defined. One could use $\mathcal{EN}(P)$, but also $\mathcal{N}(P)$ as one of the indicators for shape.
Another  Morse theorectic approach to the shape of a tetrahedron can be found in  \cite{Siersma}.

\subsection*{Acknowledgements
}The authors acknowledge useful discussions with George Khimshiashvili and Maciej Denkowski. The authors are especially grateful to  CIRM, Luminy, where this research was initiated in framework of ”Research in Pairs”
program in January, 2024.  {Ivan Nasonov was supported by the  social investment program "Native towns".}


\end{document}